\documentclass[12pt]{amsart}
\usepackage{amssymb,mathrsfs,longtable}
\usepackage[matrix,arrow,curve]{xy}
\usepackage{setspace}
\usepackage{multirow}
\usepackage{textcomp}
\usepackage{hyperref}

\usepackage[dvipsnames]{xcolor}
\usepackage{graphicx}

\newcounter{cequation}[section]

\newtheorem{theorem}[cequation]{Theorem}
\newtheorem*{theorem*}{Theorem}
\newtheorem{lemma}[cequation]{Lemma}
\newtheorem{corollary}[cequation]{Corollary}

\newtheorem{proposition}[cequation]{Proposition}

\theoremstyle{definition}
\newtheorem{example}[cequation]{Example}
\newtheorem{definition}[cequation]{Definition}
\newtheorem*{definition*}{Definition}
\newtheorem{question}[cequation]{Question}
\newtheorem{problem}[cequation]{Problem}
\newtheorem*{notation*}{Notation}

\theoremstyle{remark}
\newtheorem{remark}[cequation]{Remark}

\makeatletter\@addtoreset{equation}{section}
\makeatletter\@addtoreset{section}{part}

\makeatother

\def \O {\mathcal{O}}
\def \CC {\mathbb{C}}
\def \P {\mathbb{P}}
\def \PP {\mathbb{P}}
\def \QQ {\mathbb{Q}}
\def \Dbcoh {\mathbf{D}^b}
\def \Pic {\mathrm{Pic}\,}
\def \hh {\mathrm{h}}
\def \Bs {\mathrm{Bs}\,}

\def \ge {\geqslant}
\def \le {\leqslant}

\title{Hodge level for weighted complete intersections}

\author{Victor Przyjalkowski and Constantin Shramov}

\address{\emph{Victor Przyjalkowski}
\newline
\textnormal{Steklov Mathematical Institute of RAS, 8 Gubkina street, Moscow 119991, Russia.
}
\newline
\textnormal{National Research University Higher School of Economics, Laboratory of Mirror Symmetry, NRU HSE, 6 Usacheva street, Moscow, 117312, Russia.
}
\newline
\textnormal{\texttt{victorprz@mi.ras.ru, victorprz@gmail.com}}}

\address{\emph{Constantin Shramov}
\newline
\textnormal{Steklov Mathematical Institute of RAS,
8 Gubkina street, Moscow 119991, Russia.
}
\newline
\textnormal{National Research University Higher School of Economics, Laboratory of Algebraic Geometry, NRU HSE, 6 Usacheva str., Moscow, 117312, Russia.
}
\newline
\textnormal{\texttt{costya.shramov@gmail.com}}}

\thanks{
Both authors were partially supported by Young Russian Mathematics award.
Victor Przyjalkowski was supported by
Laboratory of Mirror Symmetry NRU HSE, RF government  grant, ag. \textnumero 14.641.31.0001.
Constantin Shramov was supported by the
HSE University Basic Research Program,
Russian Academic Excellence Project~\mbox{``5-100''},
and by the Foundation for the
Advancement of Theoretical Physics and Mathematics ``BASIS''}

\begin{document}

\begin{abstract}
We give lower bounds for Hodge numbers of smooth well formed Fano weighted complete intersections.
In particular, we compute their Hodge level, that is, the maximal distance between non-trivial Hodge numbers in the same row
of the Hodge diamond.
This allows us to classify varieties
whose Hodge numbers are like that of a projective space, of a curve, or of a Calabi--Yau variety of low dimension.
\end{abstract}

\maketitle

\section{Introduction}

Let $X$ be a smooth projective $n$-dimensional variety defined over the field of complex numbers.
Hodge numbers of $X$ are among its most basic invariants. Thus it is interesting to understand the situation when
they satisfy some minimality conditions.

\begin{definition}
We say that $X$ is \emph{$\QQ$-homologically minimal},
if its Hodge numbers are as small as possible, that is,
the same as those of the $n$-dimensional projective space.
\end{definition}

An old problem is to characterize varieties that satisfy a stronger property
(for which being $\QQ$-homologically minimal is a necessary condition): namely, varieties whose
cohomology ring over~$\mathbb{Z}$ is isomorphic to that of~$\P^n$. Fujita \cite{Fujita} showed that a smooth
Fano variety of dimension at most five with the latter
property is actually isomorphic to~$\P^n$. Hirzebruch and Kodaira \cite{HirzebruchKodaira}
proved the same result under different additional conditions;
see also the work \cite{KobayashiOchiai} of Kobayashi and Ochiai for another result of this kind.

As for varieties that are just $\QQ$-homologically minimal,
a smooth odd-dimensional quadric gives an example of such a variety that is not isomorphic to a projective space.
It was checked by Ewing and Moolgavkar~\cite{EM} that there are no other $\QQ$-homologically
minimal varieties among smooth complete intersections (in usual projective spaces).
One can see from the classification of smooth del Pezzo surfaces and Fano threefolds (see~$\mbox{\cite[\S12.2]{IP99}}$) that
in dimension $2$ the only variety like this is $\P^2$, while in dimension $3$ there are four (families of) examples:
the projective space $\P^3$, the quadric, the del Pezzo threefold~$V_5$ of anticanonical degree $40$, and a Fano variety
$V_{22}$ of Fano index $1$ and anticanonical degree $22$; while the first three of these threefolds are unique in their
deformation classes, the fourth family is $6$-dimensional.
Wilson \cite{Wil86} proved that in dimension
$4$ every smooth $\QQ$-homologically minimal Fano variety is isomorphic to $\P^4$, with a possible (unlikely) exception of
varieties with certain explicitly described properties.
There are further examples of $\QQ$-homologically minimal fivefolds (see below).

Another property that implies $\QQ$-homological minimality of $X$ is the existence in the derived category $\Dbcoh(X)$ of coherent sheaves on $X$ of a full exceptional
collection of the minimal possible length, that is, of length $n+1$;
the fact that $n+1$ is indeed the minimal possible length
follows from additivity of Hochschild
homology under semi-orthogonal decompositions,
see~\cite[Corollary~7.5]{Kuz09},
and Hochschild--Kostant--Rosenberg theorem.
It is known that such exceptional collection exists for~$\P^n$ itself (see~\cite{Be78});
for odd dimensional-quadrics (see~\cite{Ka88});
for the varieties $V_5$ (see~\cite{Or91}) and $V_{22}$ (see~\cite{Kuz96}).
It was proved in~\cite{GKMS13} that the only Fano fourfold $X$ with a full exceptional collection
{of length~$5$} in~\mbox{$\Dbcoh(X)$} is~$\P^4$. There are three more examples of Fano fivefolds that have full exceptional
collections {of length~$6$} in derived categories of coherent sheaves, see~\mbox{\cite[\S6]{Kuz06}} (cf.~also~\mbox{\cite[\S1.4]{Kuz18}});
in particular, these varieties are $\QQ$-homologically minimal. A conjecture of Bondal
and Orlov
predicts that a Fano variety
$X$ of even dimension $n$  with a full exceptional collection in $\Dbcoh(X)$ of length $n+1$ is isomorphic to a projective space.

\begin{remark}
In~\cite{Wil86} and~\cite{GKMS13}, different terminology was used: $\QQ$-homologically minimal varieties were called \emph{$\QQ$-homology projective spaces}.
\end{remark}

A slightly more general situation is described by the following definition.

\begin{definition}
We say that $X$ is \emph{diagonal},
if one has $h^{p,q}(X)=0$ unless $p=q$.
\end{definition}

It is known  that if $\Dbcoh(X)$
contains a full exceptional collection, then~$X$ is diagonal, due to additivity of Hochschild homology as above.
Examples of diagonal varieties include even-dimensional quadrics and even-dimensional intersections of two quadrics.
For varieties of the latter two classes, there is indeed a full exceptional collection in
the derived category of coherent sheaves, see~\cite{Ka88} and~\mbox{\cite[Corollary~5.7]{Kuz08}}.
The following natural question is attributed to Bondal
and Orlov: is it true that for every diagonal Fano
variety~$X$ there is a full exceptional collection
in $\Dbcoh(X)$?
It appeared that the answer to this question is negative:
a counterexample was recently constructed by Kuznetsov~\cite{Kuz18b}.
However, there is Lunts' refinement of this question requiring, in addition, that $X$ is Tate,
see~\cite[Conjecture 1.2]{EL16} (note that the requirement for $X$ to be Fano is not necessary in this case).
An alternative refinement is a requirement for {the Grothendieck group} $K_0(X)$ of the variety $X$
to be a free group.
Besides that, there is another conjecture of Orlov
predicting that every variety with a full exceptional collection is rational.

There are further classes of varieties that may be considered relatively simple from
the point of view of derived categories of coherent sheaves.
One of the interesting situations is when $\Dbcoh(X)$ contains
an exceptional collection such that its semi-orthogonal complement in $\Dbcoh(X)$
is a category $\mathcal{A}_X$ with certain nice properties. For instance, $\mathcal{A}_X$ may itself
be equivalent to the derived category of coherent sheaves on a variety of small dimension, or on
a Calabi--Yau variety.
One can easily see that if $\mathcal{A}_X$ is equivalent to the derived category
of coherent sheaves on a curve,
then again by additivity of Hochschild homology and Hochschild--Kostant--Rosenberg theorem
for $X$ we obtain strong restrictions on the Hodge numbers of $X$.
Namely, under an additional assumption that $h^{p,q}(X)=0$ unless~\mbox{$p=q$} or $p+q=n$,
the variety $X$ is of curve type in the following sense.

\begin{definition}
We say that $X$ is \emph{of curve type},
if $n$ is odd,
and one has $h^{p,q}(X)=0$ unless $p=q$ or
$\{p,q\}=\{\frac{n-1}{2},\frac{n+1}{2}\}$.
\end{definition}

Similarly ,one can easily see that if $\mathcal{A}_X$ is equivalent to the derived category
of coherent sheaves on a a Calabi--Yau variety of dimension $m$,
then $X$ is of $m$-Calabi--Yau type in the following sense.

\begin{definition}
Given a positive integer $m$, we say that $X$ is \emph{of $m$-Calabi--Yau type},
if~$n$ has the same parity as $m$, one has  $h^{\frac{n-m}{2},\frac{n+m}{2}}(X)=1$, and
$h^{p,q}(X)=0$ for $p<\frac{n-m}{2}$ or~\mbox{$p=\frac{n-m}{2}$}, $p+q\neq n$.
We say that $X$ is \emph{of K3 type},
if it is of $2$-Calabi--Yau type.
\end{definition}

In other words, $X$ is of $m$-Calabi--Yau type if for the Hochschild homology of $\Dbcoh(X)$ one has $HH_s(\Dbcoh(X))=0$ for $s<-m$ and $s>m$, and
$$
\dim HH_{-m}(\Dbcoh(X))=\dim HH_{m}(\Dbcoh(X))=1.
$$

There is a notion of $m$-Calabi--Yau category defined in terms of the Serre functor, see, for example,~\mbox{\cite[Definition~1.1]{Kuz15b}}.
Derived categories of $m$-dimensional non-commutative Calabi--Yau varieties (see~\cite[\S 4.4]{Kuz15b}) are examples of categories of this type.
Note that if $m$-Calabi--Yau category $\mathcal T$ is a derived category of an $m$-dimensional variety,
one has $HH_s(\mathcal T)=0$ for $s<-m$ and $s>m$. However such vanishing is not expected to hold for an arbitrary
$m$-Calabi--Yau category.

\begin{question}
\label{question:CY type}
Does the vanishing hold in the geometric case, that is, when $\mathcal T=\mathcal A_X$
(cf.~Example~\ref{example:mCY cubics} below)?
\end{question}

Examples of varieties of curve type are given by Fano threefolds;
for a Fano threefold $X$ a category $\mathcal A_X$
is indeed of curve type (that is, its only possibly non-vanishing Hochschild
homology groups for $\mathcal A_X$ are of
degrees~$-1$,~$0$, and~$1$),
although for some Fano threefolds this category is
not {equivalent to} derived category of coherent sheaves on an actual curve.
An example of a variety of curve type whose derived category of coherent sheaves contains the derived category of
coherent sheaves on a curve as a semi-orthogonal complement to an exceptional collection
is given by an odd-dimensional intersection of two quadrics, see~\mbox{\cite[Corollary~5.7]{Kuz08}}.
A cubic fourfold in $\P^5$ is a variety of K3 type; its derived category of coherent sheaves
contains a derived category of a non-commutative K3 surface
as a semi-orthogonal complement to an exceptional collection by~\cite{Kuz10}.
Examples of varieties of $3$-Calabi--Yau type are given by a smooth five-dimensional quartic hypersurface in the weighted projective space~\mbox{$\P(1^6,2)$}, see~\S\ref{section:preliminaries} below for definitions and references
concerning weighted projective spaces and subvarieties therein;
a smooth five-dimensional complete intersection of a quadric and a cubic in~$\P^7$;
and a smooth seven-dimensional cubic in $\P^8$ (see~\mbox{\cite[\S3.1]{IM15}}).
The derived categories of coherent sheaves on these varieties contain
categories of $3$-Calabi--Yau type as semi-orthogonal complements to
exceptional collections, see~\mbox{\cite[Proposition~4.6]{IM15}}.

In this paper we classify
$\QQ$-homologically minimal varieties, diagonal varieties, varieties of curve type,
varieties of K3 type, and varieties of $3$-Calabi--Yau type
among smooth Fano
weighted complete intersections, see~\S\ref{section:preliminaries} for precise definitions.
The reason to consider this class of varieties is as follows.
One of the main ways to construct examples of Fano varieties is to describe them as complete intersections in already
known ones,
for instance, in Grassmannians, or in toric varieties whose properties are somewhat close to those of the usual projective space. 
Computing Hodge numbers for varieties of the former type is an interesting problem;
some partial results in this direction were obtained in~\cite{FM}.
On the other hand, one may argue that toric varieties whose properties are most close to those
of a projective space are weighted projective spaces. Fortunately, in this case we also have a nice and efficient
way to compute Hodge numbers of a complete intersection.

Let $X$ be a smooth complete intersection of dimension $n$ in a $\QQ$-factorial toric variety~$Y$.
From the Lefschetz-type theorem (see~\cite[Proposition~1.4]{Ma99})
it follows that~\mbox{$h^{p,q}(X)=h^{p,q}(Y)$} for $p+q\neq n$ and~\mbox{$h^{p,q}(X)\ge h^{p,q}(Y)$} for $p+q=n$. Moreover, if~$Y$ is a weighted projective space,
then~\mbox{$h^{p,p}(X)=1$} for $p\neq \frac{n}{2}$. This means that the only
``non-trivial'' Hodge numbers of $X$ are the middle ones $h^{p,n-p}(X)$.
Thus tho check that the variety $X$ is $\QQ$-homologically minimal, diagonal,
 of curve type, or of $m$-Calabi--Yau type,
it is enough to check the relevant conditions for the middle raw of the Hodge diamond.

In this paper we
provide lower bounds for some middle Hodge numbers of smooth Fano weighted complete intersections. For this we use a well known
method to compute Hodge numbers of weighted complete intersections as dimensions of graded
components of some bigraded ring, see Theorem~\ref{theorem:middle-Hodge-numbers} below.
As a corollary we describe all $\QQ$-homologically minimal and diagonal smooth Fano weighted complete intersections,
as well as ones of curve type and of $2$- and $3$-Calabi--Yau types.
To make a long story short, there are no new varieties of these five types except for examples mentioned in the above discussion.

Our first result is a complete classification of smooth Fano weighted complete intersections that are
$\QQ$-homologically minimal, diagonal, or of curve type. {We refer the reader to~\cite{Do82}
and~\cite{IF00}, or to~\S\ref{section:preliminaries} below, for relevant definitions.} We will exclude the case of the projective space (that can be thought of as a complete
intersection of codimension $0$ in itself) from our considerations to simplify notation.

\begin{theorem}\label{theorem:main}
Let $X$ be a smooth well formed Fano weighted complete intersection of dimension $n$ which
is not an intersection with a linear cone. The following assertions hold.
\begin{itemize}
\item[(i)] The variety $X$ is $\QQ$-homologically minimal if and only if $X$ is an odd-dimensional quadric in~$\P^{n+1}$.

\item[(ii)] The variety $X$ is diagonal
if and only if
either $n=2$ (cf. Table~\ref{table:dim 2}),
or $X$ is a quadric in $\P^{n+1}$, or $X$ is an even-dimensional intersection of
two quadrics in~$\P^{n+2}$.

\item[(iii)] The variety $X$ is of curve type if and only if either $n=1$, or $n= 3$ (cf. Table~\ref{table:dim 3}),
or $X$ is an odd-dimensional complete intersection of $k\le 3$ quadrics in~$\P^{n+k}$, or~$X$ is a five-dimensional cubic in~$\P^6$.
\end{itemize}
\end{theorem}

Our second result concerns smooth Fano weighted complete intersections that are of $m$-Calabi--Yau type for some $m$.
For a weighted complete intersection $X$ of multidegree~\mbox{$(d_1,\ldots,d_k)$} in~\mbox{$\P(a_0,\ldots,a_N)$} we put
$$
i_X=\sum a_i-\sum d_j.
$$

\begin{theorem}\label{theorem:main-CY}
Let $X$ be a smooth well formed Fano weighted complete intersection
of multidegree $(d_1,\ldots,d_k)$, $d_1\le\ldots\le d_k$,
which is not an intersection with a linear cone.
Put $n=\dim X$, and let $m$ be a positive integer.
Then $X$ is of
$m$-Calabi--Yau type if and only if the following conditions hold:
\begin{itemize}
\item one has either $k=1$ or $d_{k-1}<d_k$;

\item the number
$i_X$ is divisible by $d_k$;

\item one has $m=n-\frac{2i_X}{d_k}$.
\end{itemize}
\end{theorem}

\begin{proposition}\label{proposition:main-CY}
Let $X$ be a smooth well formed Fano weighted complete intersection
which is not an intersection with a linear cone.
The following assertions hold.
\begin{itemize}
\item[(i)] The
variety $X$ is of K3 type
if and only if $X$ is
a four-dimensional cubic in $\P^5$.

\item[(ii)] The variety $X$ is of $3$-Calabi--Yau type if and only if $X$ is
either a five-dimensional quartic hypersurface in $\P(1^6,2)$,
or a five-dimensional complete intersection of a quadric and a cubic in $\P^7$,
or a seven-dimensional cubic in $\P^8$.
\end{itemize}
\end{proposition}

Using the method of the proof of Theorem~\ref{theorem:main-CY}, one can also classify smooth
Fano complete intersections of $m$-Calabi--Yau type for any given $m$.

The following example related to Theorem~\ref{theorem:main-CY} is well known to experts.

\begin{example}[cf. Question~\ref{question:CY type}]
\label{example:mCY cubics}
Let $X$ be a smooth cubic in $\P^{n+1}$. Then $X$ is of $m$-Calabi--Yau type if and only if~\mbox{$n=3m-2$}.
By~\cite[Corollary~4.2]{Kuz04} (or by~\cite[Corollary~4.3]{Kuz04}, see also~\mbox{\cite[Corollary~4.2]{Kuz15b}})
the derived category of coherent sheaves on a cubic of dimension $3m-2$ contains an $m$-Calabi--Yau
category as a semi-orthogonal complement to
an exceptional collection.
By~\cite[Corollary~4.2]{Kuz04} the same assertion holds for other smooth hypersurfaces in weighted projective spaces
satisfying the conditions of Theorem~\ref{theorem:main-CY}, e.g., for
$(2m-1)$-dimensional quartic hypersurfaces in~\mbox{$\P(1^{2m}, 2)$}.
It would be interesting to know if this is also the case for complete intersections of larger codimension that
satisfy the restrictions provided by Theorem~\ref{theorem:main-CY}, e.g., for $(3m-4)$-dimensional complete intersections of
a quadric and a cubic in~$\P^{3m-2}$.
\end{example}

\begin{remark}\label{remark:Iliev-Manivel}
Fano varieties of $3$-Calabi--Yau type were considered {(and were called just \emph{varieties of Calabi--Yau type})}
in~\cite{IM15} under a certain additional conditions.
Namely, the authors of~\cite{IM15} required that for
an $n$-dimensional {Fano} threefold $X$ {of $3$-Calabi--Yau type}
one has $h^{p,0}(X)=0$ for all $p$ and
for any generator $\omega \in H^{n+2,n-1}(X)$ the contraction map
\[\xymatrix{
H^1(X, TX) \ar[r]^\omega & H^n(X, \Omega^{n+1}_X)
}\]
is an isomorphism, see condition~(2) in~\cite[Definition~2.1]{IM15} (note that there is a misprint in the index of the target cohomology group
in \cite[Definition~2.1]{IM15}).
Proposition~\ref{proposition:main-CY}(ii)
shows that the examples found in~\cite[\S3.1]{IM15}
give all possible varieties of this type among smooth Fano weighted complete intersections.
\end{remark}

Keeping in mind the known results
on the derived categories of coherent sheaves on varieties and exceptional collections
therein (see~\cite{KO95}~\cite{Ka88},~\cite{Kuz08}, \mbox{\cite[Corollary~4.2]{Kuz15b}}),
we conclude that Theorem~\ref{theorem:main} and
Proposition~\ref{proposition:main-CY}
imply the following assertion.

\begin{corollary}
\label{corollary:CY categories}
Let $X$ be a smooth well formed Fano weighted complete intersection of dimension $n$ which
is not an intersection with a linear cone. Then the following assertions hold.
\begin{itemize}
\item[(i)] There is a full exceptional collection in the derived category $\Dbcoh(X)$
if and only if {either~\mbox{$n=2$}, or~$X$ is a quadric in $\P^{n+1}$, or $X$ is} an even-dimensional intersection of
two quadrics in~$\P^{n+2}$.

\item[(ii)] The derived category $\Dbcoh(X)$ is of curve type
if and only if $X$ is either  a threefold, or an odd-dimensional complete intersection of $k\le 3$ quadrics in~$\P^{n+k}$, or a five-dimensional cubic in~$\P^6$.

\item[(iii)] The derived category $\Dbcoh(X)$ contains a derived category of a non-commutative K3 surface
as a semi-orthogonal complement to an exceptional collection
if and only if $X$ is
a four-dimensional cubic in $\P^5$.
\end{itemize}

\end{corollary}

\begin{remark}
If $X$ is
either a five-dimensional quartic hypersurface in~\mbox{$\P(1^6,2)$},
or a five-dimensional complete intersection of a quadric and a cubic in~$\P^7$,
or a seven-dimensional cubic in $\P^8$, then the derived category $\Dbcoh(X)$ contains a subcategory of $3$-Calabi--Yau type
as a semi-orthogonal complement to an exceptional collection, see~\mbox{\cite[Proposition~4.6(3)]{IM15}}.
It would be intersting to find out if the converse is also true.
\end{remark}

The notions we have introduced above are related to the following classical notion.

\begin{definition}[{cf. \cite[\S1]{Rapoport}, \cite[\S2a]{Carlson}}]
\label{definition: Hodge level}
Let $X$ be a smooth projective variety of dimension $n$.
Put
$$
\hh(X)=\max\{q-p\mid h^{p,q}(X)\neq 0\}.
$$
The number $\hh(X)$ will be called \emph{the Hodge level} of~$X$.
\end{definition}

Note that $\hh(X)$ is always non-negative.
In the terminology of \cite{Rapoport}, the number $\hh(X)$ is the maximal Hodge level
of the Hodge structures on the cohomology groups $H^r(X,\mathbb{Z})$ for $0\le r\le 2\dim X$.
If $X$ is an $n$-dimensional smooth weighted complete intersection, then $\hh(X)$ equals
the Hodge level of the Hodge structure on $H^n(X,\mathbb{Z})$.
It is obvious that~\mbox{$\hh(X)=0$} if and
only if $X$ is diagonal; in particular, this holds for any smooth quadric.
One has $\hh(X)\le 1$ if $X$ is of curve type.
If $X$ is of $m$-Calabi--Yau type (for instance, if $X$ is a Calabi--Yau variety of dimension~$m$),
then~\mbox{$\hh(X)=m$}. Smooth complete intersections~$X$ in a usual projective space
such that~\mbox{$\hh(X)\le 1$} were classified in~\mbox{\cite[\S2]{Rapoport}}.

Our next result describes
some general properties of Hodge level for weighted complete intersections.
For a weighted complete intersection $X$ of multidegree~\mbox{$(d_1,\ldots,d_k)$} in~\mbox{$\P(a_0,\ldots,a_N)$},
where~\mbox{$d_1\le\ldots\le d_k$}, we put
$$
p_X=\left\lceil\frac{i_X}{d_k}\right\rceil.
$$

\begin{proposition}\label{proposition:Hodge-level}
Let $X$ be a smooth well formed Fano weighted complete intersection of dimension $n$ which
is not an intersection with a linear cone.
If $X$ is an odd-dimensional quadric, then $\hh(X)=0$.
Otherwise~\mbox{$\hh(X)=n-2p_X$}.
\end{proposition}

Note that if $X$ is a Fano variety of dimension $n$,
then~\mbox{$h^{0,n}(X)=0$} by Kodaira vanishing, so that~\mbox{$\hh(X)\le n-2$}.
Proposition~\ref{proposition:Hodge-level} implies the following.

\begin{corollary}\label{corollary:Hodge-level}
Let $X$ be a smooth well formed Fano weighted complete intersection of dimension $n\ge 2$ which
is not an intersection with a linear cone.
The following assertions hold.
\begin{itemize}
\item[(i)]
Suppose that either $i_X\le 2$,
or $i_X\le 3$ and $X$ is not a complete intersection of quadrics in a projective space,
or $i_X\le 4$ and $X$ is not a complete intersection of quadrics and cubics in a projective space.
Then~\mbox{$\hh(X)=n-2$}.

\item[(ii)] Suppose that $X$ is not a complete intersection of
quadrics in a projective space. Then~\mbox{$\hh(X)\ge\frac{n-4}{3}$}.
\end{itemize}
\end{corollary}

An immediate consequence of Corollary~\ref{corollary:Hodge-level}(ii)
is the following.

\begin{corollary}
For every number $h$ the dimension of
smooth well formed Fano weighted complete intersections~$Y$ such that~$Y$ is
not an intersection with a linear cone, not a complete intersections of
quadrics in a projective space, and~\mbox{$\hh(Y)<h$}, is bounded.
\end{corollary}

\begin{remark}
\label{remark:quasismooth}
Our results are motivated by applications to the derived categories of coherent sheaves on smooth varieties.
However, if we allow singularities,
then the cohomological dimension of the derived category of coherent sheaves on a variety becomes infinite, and
it's hard to study its semiorthogonal components of K3 type. On the other hand, if a weighted complete intersection $X$
is quasi-smooth (see Definition~\ref{definition: quasi-smoothness}), then one can consider
the derived category of a smooth stack $\mathcal X$ with support at $X$.
Another approach is to consider  the subcategory $Perf(X)\subset \Dbcoh(X)$ of perfect complexes.
Anyway, cohomology groups of quasi-smooth complete intersections admit pure Hodge structures (see~\mbox{\cite[\S11]{BC94}}), so one can ask
which of them are of curve, K3, or $m$-Calabi--Yau type. Moreover, the approach to check this
used in our paper can be applied to quasi-smooth weighted complete intersections.

However it turns out that the quasi-smoothness restriction
is too weak to get a reasonable classification. It is easy to see that any (well formed) weighted projective space
is a quasi-smooth $\QQ$-homologically minimal variety (which can be considered as a weighted complete intersection
of codimension $0$ in itself). Furthermore, calculations of V.\,Alexeev show that
there are $124$ families of quasi-smooth well formed Fano hypersurfaces of K3 type in five-dimensional weighted projective spaces with weights up to~$50$;
there are $122$ families of quasi-smooth well formed Fano hypersurfaces of K3 type in seven-dimensional weighted projective spaces with weights up to $30$;
there are $105$ families of quasi-smooth well formed Fano hypersurfaces of K3 type in nine-dimensional weighted projective spaces with weights up to $20$, etc.

A nice observation due to A.\,Kuznetsov is that most of hypersurfaces from Alexeev's list have birational
transformations to varieties related to K3 surfaces,
so categories of K3 surfaces
naturally appear in their derived categories.
\end{remark}

\medskip

The paper is organized as follows.
In~\S\ref{section:preliminaries} we give the relevant definitions, recall the method to compute Hodge numbers of weighted complete intersections
and prove some auxiliary results.
In~\S\ref{section:bounds} we obtain the bounds for Hodge numbers of Fano weighted complete intersections.
In~\S\ref{section:proofs} we prove the main results of the paper, namely,  Theorems~\ref{theorem:main}
and~\ref{theorem:main-CY}, Propositions~\ref{proposition:main-CY} and~\ref{proposition:Hodge-level},
and Corollary~\ref{corollary:Hodge-level}.
In~\S\ref{section:quasi-smooth} we briefly discuss the quasi-smooth case. In~\S\ref{section:discussion} we discuss some open questions.
In the appendix we provide the well known lists of two- and three-dimensional
smooth well formed Fano weighted complete intersections.

\medskip

The authors are grateful to V.\,Alexeev, A.\,Fonarev, A.\,Iliev, A.\,Kuznetsov, L.\,Manivel, D.\,Orlov, Yu.\,Prokhorov,
and Z.\,Tur\v{c}inovi\'c
for useful discussions. Special thanks go to the referee whose numerous comments helped
to forge the present form of our paper.

\section{Preliminaries}
\label{section:preliminaries}

We recall here some basic properties of weighted complete intersections. We refer the reader
to~\cite{Do82} and~\cite{IF00} for more details. Put
$$
\P=\P(a_0,\ldots,a_N)=\mathrm{Proj}\, \CC[x_0,\ldots,x_N],
$$
where the weight of $x_i$ equals $a_i$.
Without loss of generality we assume that $a_0\le\ldots\le a_N$.
We will use the abbreviation
\begin{equation*}
(a_0^{r_0},\ldots,a_m^{r_m})=
(\underbrace{a_0,\ldots,a_0}_{r_0\ \text{times}},\ldots,\underbrace{a_m,\ldots,a_m}_{r_m\ \text{times}}),
\end{equation*}
where $r_0,\ldots,r_m$ will be allowed to be
any positive integers. If some of $r_i$ is equal to $1$ we drop it for simplicity.

The weighted projective space $\P$ is said to be \emph{well formed} if the greatest common divisor of any $N-1$ of the weights~$a_i$ is~$1$. Every weighted projective space is isomorphic to a well formed one, see~\cite[1.3.1]{Do82}.
{If $\P$ is well formed, then the singular locus of $\P$ is a union of strata
$$
\left\{(x_0:\ldots:x_N) \mid x_i=0 \text{\ for all\ } i\notin J\right\}
$$
for all subsets $J\subset \{0,\ldots,n\}$ such that the greatest common divisor of the weights~$a_i$
for~\mbox{$i\in J$} is greater than~$1$, see~\cite[5.15]{IF00}.}
A subvariety $X\subset \P$ is said to be \emph{well formed}
if~$\P$ is well formed and
$$
\mathrm{codim}_X \left( X\cap\mathrm{Sing}\,\P \right)\ge 2.
$$

We say that a subvariety $X\subset\P$ of codimension $k\ge 1$ is a \emph{weighted complete
intersection of multidegree $(d_1,\ldots,d_k)$} if its weighted homogeneous ideal in $\CC[x_0,\ldots,x_N]$
is generated by a regular sequence of $k$ homogeneous elements of degrees $d_1,\ldots,d_k$.
The above condition is equivalent to the requirement that the codimension of (every irreducible component of) the variety $X$
equals~$k$. This follows from
a similar equivalence for the variety $C_X\subset\mathbb{A}^{N+1}$ defined as the closure of the preimage of
$X$ under the natural projection $\mathbb A^{N+1}\setminus \{0\}\to \P$ and considered as an intersection of
$k$ hypersurafces in $\mathbb{A}^{N+1}$. The latter equivalence  can be deduced, for instance,
from~\mbox{\cite[Theorem II.8.21A(c)]{Ha77}} or from \cite[Exersise II.8.4]{Ha77}.

Note that $\P$ can be thought of as a complete
intersection of codimension $0$ in itself (which gives us a nice smooth Fano variety if~\mbox{$\P\cong\P^N$}),
but we do not consider this case for simplicity.
The weighted complete intersection~$X$ is said to be \emph{an intersection
with a linear cone} if one has $d_j=a_i$ for some~$i$ and~$j$.
In this case one can exclude the $i$-th weighted homogeneous coordinate and think about
$X$ as a weighted complete intersection in a weighted projective space of lower dimension, provided that
$X$ is general enough, cf.~\mbox{\cite[Remark~5.2]{PrzyalkowskiShramov-Weighted}}.

We will be interested in {smooth} well formed weighted
complete intersections. {Note that by~\cite[Corollary~2.14]{PrzyalkowskiShramov-Weighted} such} varieties are automatically quasi-smooth
{(cf. Definition~\ref{definition: quasi-smoothness}). This allows us to use many auxiliary results that
were proved for quasi-smooth weighted complete intersections.}

For a smooth well formed weighted complete intersection, one has the following relation
between the weights $a_i$ and the degrees $d_j$.

\begin{lemma}[{see~\cite[Lemma~2.15]{PrzyalkowskiShramov-Weighted}, cf. \cite[6.12]{IF00}, \cite[Proposition~3.1]{PST17}}]
\label{lemma:a-vs-d}
Let~\mbox{$X\subset\P$} be a smooth well formed weighted complete intersection
of multidegree~\mbox{$(d_1,\ldots,d_k)$}.
Then for every~$1\le t\le k$ and every choice of~$t$ weights~\mbox{$a_{i_1},\ldots,a_{i_t}$, $i_1<\ldots<i_t$},
such that their greatest common divisor~$\delta$ is greater than~$1$
there exist~$t$ degrees~\mbox{$d_{s_1},\ldots,d_{s_t}$, $s_1<\ldots<s_t$},
such that their greatest common divisor is divisible by~$\delta$.
\end{lemma}

Consider the sheaf~\mbox{$\O_{\P}(1)$},
see~\mbox{\cite[1.4.1]{Do82}}. Recall that~\mbox{$\O_{\P}(1)$}
is usually not invertible. However, if $\P$ is well formed,
the restriction of~\mbox{$\O_{\P}(1)$} to $\P\setminus \mathrm{Sing}\,\P$
is an invertible sheaf, see~\mbox{\cite[1.5.5]{Do82}}.
Let $X$ be a weighted complete intersection in~$\P$.
Denote by $\O_X(1)$ the restriction of~\mbox{$\O_{\P}(1)$} to~$X$.
If $X$ is smooth and well formed,
then $X$ is contained in the smooth locus of~$\P$, see for instance~\mbox{\cite[Proposition~2.11]{PrzyalkowskiShramov-Weighted}}. Hence
the sheaf $\O_X(1)$ is a line bundle on $X$ in this case.

\begin{lemma}[{\cite[Remark~4.2]{Okada2}, \cite[Proposition~2.3]{PST17}}]
\label{lemma:Okada}
Let $X$ be a smooth well formed weighted complete intersection of dimension at least three in~$\P$.
Then the class of the line bundle $\mathcal{O}_{\P}(1)\vert_X$ is not divisible in~$\Pic(X)$.
\end{lemma}

For a weighted complete intersection $X$
of multidegree $(d_1,\ldots,d_k)$ in $\P$ we denote
$$
i_X=\sum a_i-\sum d_j.
$$
It is easy to describe the canonical class of a weighted complete intersection.

\begin{theorem}[{see~\cite[Theorem 3.3.4]{Do82}, \cite[6.14]{IF00}}]
\label{theorem:adjunction}
Let $X$ be a smooth
well formed weighted complete intersection of multidegree $(d_1,\ldots,d_k)$
in~$\P$.
Then
$$
\omega_X\simeq \O_X\left(-i_X\right).
$$
\end{theorem}

In this paper we will be mostly concerned with Fano weighted complete intersections.
There are various results bounding the relevant parameters.
Recall from~\cite[Corollary~5.3(i)]{PST17} that {if there exists a smooth well formed Fano weighted
complete intersection in $\P$, then $a_{0}=1$. In this case we will define} the number $0\le {l_{\PP}}\le N$ by the conditions
\begin{equation}\label{eq:l}
1=a_0=\ldots=a_{l_{\PP}}<a_{{l_{\PP}}+1}
\end{equation}
if $a_N>1$,
and put ${l_{\PP}}=N$ otherwise; thus ${l_{\PP}}+1$ is the number of weights among $a_i$'s that are equal to~$1$.

\begin{theorem}\label{theorem:Fano-invariants}
Let $X$ be a smooth well formed Fano weighted complete intersection
of codimension $k$ and dimension $n=N-k$
in $\P$ which is not an intersection with a linear cone.
Then the following inequalities hold:
\begin{itemize}
\item[(i)] $a_N\le N$;

\item[(ii)] $k\le n$;

\item[(iii)] ${l_{\PP}}\ge k$.
\end{itemize}
\end{theorem}
\begin{proof}
Assertion~(i) is proved in~\cite[Theorem~1.1]{PrzyalkowskiShramov-Weighted}.
Assertion~(ii) is~\cite[Theorem~1.3]{ChenChenChen}.
Assertion~(iii) follows from~\cite[Corollary~5.3(i)]{PST17}.
\end{proof}

Using the bounds provided by Theorem~\ref{theorem:Fano-invariants},
one can easily obtain the well known lists of all smooth Fano weighted complete intersections of small dimensions.

\begin{lemma}
\label{lemma:small-dimension}
Let $X$ be a smooth well formed Fano weighted complete intersection of dimension $n$
in $\P$ which is not an intersection with a linear cone.
If $n=1$, then $X$ is a conic in~$\P^2$.
If $n=2$, then $X$ is one of the four types of del Pezzo surfaces listed
in Table~\ref{table:dim 2}.
If $n=3$, then $X$ is one of the nine types of Fano threefolds listed
in Table~\ref{table:dim 3}.
\end{lemma}

\begin{remark}\label{remark:index}
By Lemma~\ref{lemma:Okada} and Theorem~\ref{theorem:adjunction}
the number $i_X$ equals the Fano index of~$X$ provided that $X$ is a smooth Fano variety of dimension $n\ge 3$.
By Lemma~\ref{lemma:small-dimension} this is also the case
if $X$ is a del Pezzo surface, i.e. a smooth
Fano variety of dimension~\mbox{$n=2$}.
However, if $X$ is a conic in $\P^2$, then $i_X=1$, while the Fano
index of $X$ equals~$2$.
\end{remark}

It is possible to bound $i_X$ in terms of dimension of a smooth Fano weighted complete intersection~$X$.

\begin{theorem}
\label{theorem:low-coindex}
Let $X$ be a smooth well formed Fano weighted complete intersection of multidegree $(d_1,\ldots,d_k)$
and dimension $n=N-k\ge 2$ in $\P$ which is not an intersection with a linear cone. Then
\begin{itemize}
\item[(i)] $X$ is not isomorphic to $\P^n$;

\item[(ii)] one has $i_X\le n$.
\end{itemize}
\end{theorem}
\begin{proof}
By Lemma~\ref{lemma:small-dimension} we may assume that $n\ge 3$, so that by Remark~\ref{remark:index}
the Fano index of $X$ equals $i_X$.
Suppose that $X\cong\P^n$, so that $i_X=n+1$.

Recall that $a_0\le\ldots\le a_N$.
We may also assume that $d_1\le\ldots\le d_k$.
Then one has~\mbox{$d_{k-i}>a_{N-i}$} for all $0\le i\le k-1$,
see~\cite[Lemma~3.1(i),(ii)]{PrzyalkowskiShramov-Weighted}.
On the other hand, for the number ${l_{\PP}}$ defined by~\eqref{eq:l} one has
$$
{l_{\PP}}\ge i_X=n+1=N-k+1
$$
by~\cite[Corollary~5.11]{PST17}.
This implies
\[
\prod\limits_{j=1}^k d_j>\prod\limits_{i=N-k+1}^N a_i=\prod\limits_{i=0}^N a_i.
\]

Now let $H$ be the class of the line bundle $\O_X(1)$ in~\mbox{$\Pic(X)$}.
Then $H$ is not divisible in~\mbox{$\Pic(X)$} by Lemma~\ref{lemma:Okada}, which means that $H$ is the class of a
hyperplane on~$\P^n$.
Hence, one has
$$
1=H^n=\frac{\prod_{j=1}^k d_j}{\prod_{i=0}^N a_i}>1,
$$
which gives a contradiction. This proves assertion~(i).

To deduce assertion~(ii) from~(i), recall that the Fano index of an arbitrary smooth $n$-dimensional Fano variety
$Y$ is bounded
by $n$ provided that $Y$ is not isomorphic to $\P^n$, see for instance~\cite[Corollary 3.1.15]{IP99}.
\end{proof}

\smallskip

Put
$$
S=\CC[x_0,\ldots,x_{N}, w_1,\ldots,w_k].
$$
Let $f_1,\ldots,f_k$ be polynomials in $\CC[x_0,\ldots,x_{N}]$ of weighted degrees $d_1,\ldots,d_k$
that generate the weighted homogeneous ideal of $X$. Let
$$
F=F(f_1,\ldots,f_k)=
w_1f_1+\ldots+w_kf_k\in S.
$$
Denote by $J=J(F)$ the ideal in $S$ generated by
\begin{equation}\label{eq:generators-J}
\frac{\partial F}{\partial w_1},\ldots, \frac{\partial F}{\partial w_k},
\frac{\partial F}{\partial x_0},\ldots,\frac{\partial F}{\partial x_{N}}.
\end{equation}
Put
$$
R=R(f_1,\ldots,f_k)=S/J.
$$

{The algebra $S$ is bigraded by $\deg (x_s)=(0,a_s)$ and $\deg(w_j)=(1,-d_j)$,
so that $F$ is a bihomogeneous polynomial of bidegree~$(1,0)$.
Thus, the bigrading descends to the ring~$R$.}

Suppose that $X$ is smooth, and put $n=N-k=\dim X$.
Let $h_{pr}^{n-q,q}(X)$ be primitive
middle Hodge numbers of~$X$, that is,
$$h_{pr}^{p,q}(X)=h^{p,q}(X)$$ for $p\neq q$ and
$$h_{pr}^{p,q}(X)=h^{p,q}(X)-1$$ otherwise.
The following theorem will be our main tool to compute and estimate
Hodge numbers of weighted complete intersections.

\begin{theorem}[{see~\cite{Di95},~\cite{Gr69},~\cite[{Proposition~2.16}]{Na97},~\cite[{Theorem 3.6}]{Ma99}}]
\label{theorem:middle-Hodge-numbers}
One has
$$
h_{pr}^{q,n-q}(X)=\dim R_{q,-i_X}.
$$
\end{theorem}

\smallskip
Recall that the number ${l_{\PP}}$ is defined by~\eqref{eq:l}. In~\S\ref{section:bounds} we will need the following notion.

\begin{definition}
\label{definition:Fermat}
We say that a weighted complete intersection
$X\subset\P$ of codimension $k$ is
\emph{of Fermat type}
if $X$ is given by equations of the form
\begin{equation}
\label{eq:Fermat}
\aligned
\alpha_{1,0} x_0^{d_1}+&\ldots+\alpha_{1,{l_{\PP}}}x_{l_{\PP}}^{d_1}+\widehat{f}_1=0,\\
&\ldots\\
\alpha_{k,0} x_0^{d_k}+&\ldots+\alpha_{k,{l_{\PP}}}x_{l_{\PP}}^{d_k}+\widehat{f}_k=0,
\endaligned
\end{equation}
where $\widehat{f}_j$, $1\le j\le k$, are weighted homogeneous polynomials
of degree $d_j$ that depend only on variables
$x_{{l_{\PP}}+1},\ldots,x_N$.
\end{definition}

Note that in the notation of Definition~\ref{definition:Fermat} one has ${l_{\PP}}\ge k$ by Theorem~\ref{theorem:Fano-invariants}(iii).

\begin{lemma}
\label{lemma:exists-Fermat}
Suppose that there exists a smooth well formed
weighted complete intersection $X$ in $\P$.
Then there
exists a smooth well formed weighted complete intersection of Fermat type of the same multidegree
in $\P$.
\end{lemma}
\begin{proof}
Let $X$ be a smooth well formed
weighted complete intersection in $\P$ given by
equations $f_1=\ldots=f_k=0$, where the weighted degree of $f_j$ equals~$d_j$.
Write~\mbox{$f_j=g_j+\widehat{f}_j$}, where $\widehat{f}_j$
is a weighted homogeneous polynomial
of degree $d_j$ that depends only on
variables~\mbox{$x_{{l_{\PP}}+1},\ldots,x_N$}, while every monomial of $g_j$
is divisible by some of the variables~\mbox{$x_0,\ldots,x_{l_{\PP}}$}.
For every~\mbox{$1\le j\le k$} define a hypersurface~$D_j$ in~$\P$
by equation
\begin{equation}\label{eq:S}
\alpha_{j,0} x_0^{d_j}+\ldots+\alpha_{j,{l_{\PP}}}x_{l_{\PP}}^{d_j}+\widehat{f}_j=0,
\end{equation}
where $\alpha_{j,i}$ are general coefficients.

Let $X'$ be the intersection of $D_1,\ldots,D_k$, so that
$X'$ is given by equations~\eqref{eq:Fermat}.
We claim that $X'$ is a weighted complete intersection (of multidegree
$(d_1,\ldots,d_k)$), so by construction it of Fermat type.

Let $\Pi\subset\P$ be the stratum given by equations $x_{{l_{\PP}}+1}=\ldots=x_N=0$.
Put
$$
X'_j=D_1\cap\ldots\cap D_j
$$
for $0\le j\le k$, so that
$X'_0=\P$ and $X'_k=X'$.
For every $0\le t\le k$
define the variety $X_t''$ as the intersection of $X_t'$
with the stratum~$\Pi$.
One can easily see that every Weil divisor on $\P$ and hence on
$X_j'$ is $\mathbb{Q}$-Cartier.
Hence the codimension
of $X_j'$ in $X_{j-1}'$ equals the codimension
of $X_j''$ in $X_{j-1}''$.
Since $\Pi\cong\P^{l_{\PP}}$ and the coefficients $\alpha_{j,i}$ in~\eqref{eq:S} are general, the latter
codimension equals~$1$.
This means that the codimension of $X'=X_k'$ in $\P$ equals~$k$,
so that $X'$ is a weighted complete intersection.

To deduce the assertion of the lemma, check that $X'$ is well formed and smooth.
This follows from the claim that $X'$ is disjoint from the
singular locus $\Sigma$ of $\P$. Indeed, let~\mbox{$\Lambda\subset\P$} be the stratum given
by equations $x_0=\ldots=x_{l_{\PP}}=0$. Since
every monomial of~$g_j$
vanishes at $\Lambda$, we see that
$X\cap\Lambda=X'\cap\Lambda$. On the other hand,
$\Sigma$ is contained in $\Lambda$, see~\cite[5.15]{IF00}.
Since $X$ is well formed and smooth, it must be disjoint
from $\Sigma$ by~\mbox{\cite[Proposition 8]{Di86}},
which means that $X'$ is disjoint from
$\Sigma$ as well.
In particular, this implies that $X'$ is well formed.
The rest thing we need to do is to prove that $X'$ does not have singularities outside of $\Sigma$.

Let $\mathcal{D}_j$ be the linear system of
all hypersurfaces
given by equations of the form~\eqref{eq:S}, and let~$\mathcal{D}_j'$
be its restriction to $X'_{j-1}$.
We claim that if $\mathcal{D}_j'$ has a base point on~$X'_{j-1}$, then
it has a base point in $\Sigma\cap X'_{j-1}$. Indeed, suppose that
$P\in X'_{j-1}$ is a base point of~$\mathcal{D}_j'$.
Then it is also a base point of $\mathcal{D}_j$.
Obviously, the base locus $\Bs \mathcal{D}_j$
is contained in $\Lambda$.
The stratum~$\Lambda$ is itself a (possibly not well formed) weighted projective space
$$
\widehat{\P}\cong\P(a_{{l_{\PP}}+1},\ldots,a_N)
$$
with weighted homogeneous coordinates $x_{{l_{\PP}}+1},\ldots,x_N$. The restriction
of $\mathcal{D}_j$ to $\Lambda$ is the \emph{complete} linear system $\widehat{\mathcal{D}}_j$ of hypersurfaces of
weighted degree $d_j$ in $\widehat{\P}$, so $\Bs {\mathcal{D}}_j=\Bs\widehat{\mathcal{D}}_j$. Using the action of the automorphism group
of $\widehat{\P}$, we see that $\Bs \widehat{\mathcal{D}}_j$ is a (possibly empty) union $\widehat{\Lambda}$ of strata
given by the vanishing of some coordinates among $x_{{l_{\PP}}+1},\ldots,x_N$.
Thus~$\widehat{\Lambda}$,
considered as a subset of $\P$, contains a singular point of $\PP$, since
$\widehat{\Lambda}\ni P$ is nonempty.

Applying Bertini's theorem we see that singularities of a general member of $\mathcal{D}_j'$ are contained in $\Bs \mathcal{D}_j'$.
The claim above shows that if the general member is singular, it has singularity on $\Sigma$.
Now the assertion of the lemma follows from the fact that $X'=X'_k$ is disjoint from~$\Sigma$.
\end{proof}

The following simple computation concerning complete intersections in usual
projective spaces will be used in the proof of Lemma~\ref{lemma:s1<k}.

\begin{lemma}\label{lemma:dim-polynomials}
Suppose that $X\subset\P^l=\mathrm{Proj}\,\CC[x_0,\ldots,x_l]$ is a complete intersection
of hypersurfaces given by homogeneous polynomials $f_1,\ldots,f_s$. Let $V$ be the graded component of degree $r$ of the quotient
al\-geb\-ra~\mbox{$\CC[x_0,\ldots,x_l]/(f_1,\ldots,f_s)$}. Then
$$
\dim V\ge {r+l-s\choose r}.
$$
Moreover, if $r>0$, then one has
$$
\dim V\ge {r+l-s\choose r}+s.
$$
\end{lemma}
\begin{proof}
Let $d_1,\ldots,d_s$ be the degrees of the polynomials $f_1,\ldots,f_s$, respectively.
The Poincar\'e series of the complete intersection 	
$X$ is given by the well known formula
$$
P_X(z)=\frac{(1-z^{d_1})\cdot\ldots\cdot(1-z^{d_s})}{(1-z)^{l+1}}=\frac{(1+z+\ldots+z^{d_1-1})\cdot\ldots\cdot(1+z+\ldots+z^{d_s-1})}{(1-z)^{l-s+1}},
$$
see for instance \cite[Corollary~3.3]{Stanley}.
This immediately implies the equality
\begin{equation}\label{eq:dimV}
\dim V=\sum\limits_{t_1=0}^{d_1-1}\ldots\sum\limits_{t_{s}=0}^{d_{s}-1}{r-t_1-\ldots-t_{s}+l-s\choose l-s}.
\end{equation}
To obtain the first lower bound for $\dim V$ consider the summand corresponding to
$t_1=\ldots=t_{s}=0$ on the right hand side of~\eqref{eq:dimV}.
To obtain the second lower bound in the case of positive $r$ consider also the summands
corresponding to
$$
t_j=1, \quad t_1=\ldots=t_{j-1}=t_{j+1}=\ldots=t_{s}=0
$$
for all $1\le j\le s$.
\end{proof}

\section{Bounds for Hodge numbers}
\label{section:bounds}

Throughout
this section $\P=\P(a_0,\ldots,a_N)$, $a_0\le \ldots\le a_N$, denotes a well formed weighted projective space.
As in~\S\ref{section:preliminaries}, define the number~\mbox{$0\le {l_{\PP}}\le N$} by conditions~\eqref{eq:l}
if~\mbox{$a_N>1$}, and put ${l_{\PP}}=N$ otherwise; in other
words,~\mbox{${l_{\PP}}+1$} is the number of~\mbox{$x_i$'s}
of weight~\mbox{$a_i=1$}. Besides this, we use the following notation.

\begin{notation*}
Let $X$ be a smooth well formed Fano weighted
complete intersection of multidegree $(d_1,\ldots,d_k)$
in $\P$.
Suppose that $d_1\le\ldots\le d_k$.
Define $s_X$ to be the maximal index $s$ such that $d_s<d_k$;
if $d_1=\ldots=d_k$, we put $s_X=0$. For convenience we denote ${d_X}=d_k$.
\end{notation*}

\begin{lemma}\label{lemma:only-quadrics}
Let $X$ be a smooth well formed Fano weighted
complete intersection
in $\P$ which is not an intersection with a linear cone.
The following assertions hold.
\begin{itemize}
\item[(i)]
If ${d_X}=2$, then $X$ is a complete intersection of $k$ quadrics in~\mbox{$\P\cong\P^N$}.

\item[(ii)]
If ${d_X}=3$, then $X$ is a complete intersection of ${s_X}$ quadrics and $k-{s_X}$ cubics in~\mbox{$\P\cong\P^N$}.
\end{itemize}
\end{lemma}
\begin{proof}
One has
$a_i\le {d_X}$ for all $0\le i\le N$ by Lemma~\ref{lemma:a-vs-d}.
Since $X$ is not an intersection with a linear
cone, we see that $a_i\neq {d_X}$ for all $0\le i\le N$.
In particular, if ${d_X}=2$,
then we have $a_i=1$ for all $0\le i\le N$, which proves assertion~(i).

Now suppose that ${d_X}=3$. Let $d$ be the minimum of the degrees
of the defining equations of $X$.
If for some $i$ one has $a_i=2$, then one must have $d=2$ by Lemma~\ref{lemma:a-vs-d}.
This again means that $X$
is an intersection with a linear cone, which is not the case by assumption.
Thus $a_i=1$ for all $0\le i\le N$, which proves assertion~(ii).
\end{proof}

Given a smooth well formed Fano weighted
complete intersection $X$ in $\P$, we put
$$
p_X=\left\lceil\frac{i_X}{{d_X}}\right\rceil \ \ \ \text{and} \ \ \ {r_X}=p_X{d_X}-i_X,
$$
so that $i_X=p_X{d_X}$ if and only if ${r_X}=0$.
The strategy of our proof of Theorems~\ref{theorem:main} and~\ref{theorem:main-CY}
and Proposition~\ref{proposition:Hodge-level} is to find, following the notation introduced before Theorem~\ref{theorem:middle-Hodge-numbers},
a certain more or less explicit subspace in $R_{p_X,-i_X}$ to show that the corresponding
Hodge number does not vanish.

\begin{lemma}
\label{lemma:s1<k}
Let $X$ be a smooth well formed Fano weighted
complete intersection in $\P$.
One has
$$
h_{pr}^{p_X,n-p_X}(X)\ge \binom{p_X+k-{s_X}-1}{p_X}\cdot\binom{{r_X}+{l_{\PP}}-{s_X}}{{r_X}}
$$
if ${r_X}<{d_X}-1$, and
%\begin{multline*}
%h_{pr}^{p_X,n-p_X}(X)\ge\\
%\ge \binom{p_X+k-{s_X}-1}{p_X}\cdot
%\left(\binom{{d_X}-1+{l_{\PP}}-{s_X}}{{d_X}-1}+{s_X}\right)
%-({l_{\PP}}+1)\cdot \binom{p_X+k-{s_X}-2}{p_X-1}
%\end{multline*}
$$
h_{pr}^{p_X,n-p_X}(X)>\binom{p_X+k-{s_X}-1}{p_X}\cdot
\left(\binom{{d_X}-1+{l_{\PP}}-{s_X}}{{d_X}-1}+{s_X}-l_{\PP}-1\right)
$$
if ${r_X}={d_X}-1$.
\end{lemma}

\begin{proof}
Recall that the Hodge numbers are constant in smooth families of projective varieties
since they are upper semicontinuous, and their sums
are dimensions of cohomology, which are locally constant by Ehresmann
theorem.
Hence by Lemma~\ref{lemma:exists-Fermat}
one can assume that~$X$ is of Fermat type. Furthermore, we can assume that the coefficients $\alpha_{j,i}$
in the expansion of the polynomials $f_j$ as in~\eqref{eq:Fermat}
are general.

Define the ideal $J'$ in $S=\CC[x_0,\ldots,x_N, w_1,\ldots,w_k]$ as one generated by
\begin{equation*}%\label{eq:generators-J-prime}
x_{l_{\PP}+1},\ldots,x_N, w_1,\ldots,w_{s_X}, (x_0,\ldots,x_{l_{\PP}})^{d_X}, f_1,\ldots, f_{s_X},
\frac{\partial F}{\partial x_0},\ldots, \frac{\partial F}{\partial x_{l_{\PP}}},
\end{equation*}
where $(x_0,\ldots,x_{l_{\PP}})^{d_X}$ denotes the ideal generated by all monomials of degree $d_X$ in
the variables $x_0,\ldots,x_{l_{\PP}}$. Since $X$ is of Fermat type, it is straightforward to check that~\mbox{$J\subset J'$}.
Let $\bar{J}$ be the image of $J'$ in $R$, so that $\bar{J}$ is the ideal in $R$ generated by
\begin{equation*}
x_{l_{\PP}+1},\ldots,x_N, w_1,\ldots,w_{s_X}, (x_0,\ldots,x_{l_{\PP}})^{d_X},
\frac{\partial F_R}{\partial x_0},\ldots, \frac{\partial F_R}{\partial x_{l_{\PP}}},
\end{equation*}
where $F_R$ is the image of $F$ in $R$.
Set $\bar{R}=R/\bar{J}\cong S/J'$.

Denote by $\tilde{f}_j$ the polynomial obtained from $f_j$ by substituting $x_{l_{\PP}+1}=\ldots=x_N=0$;
in other words, $\tilde{f}_j$ is the ``Fermat part'' of $f_j$, and $\tilde{f}_j=f_j-\widehat{f}_j$ in the notation of Definition~\ref{definition:Fermat}.
Note that since the coefficients $\alpha_{j,i}$  are general, the polynomials $\tilde{f}_j$ define a complete intersection in the usual
projective space $\PP^{l_{\PP}}$.
Let $\tilde{F}$ be the polynomial obtained from $F$ by substituting
$$
x_{l_{\PP}+1}=\ldots=x_N=w_1=\ldots=w_{s_X}=0,
$$
so that
$$
\tilde{F}=\sum\limits_{j=s_X+1}^kw_j\tilde{f}_j.
$$
Let $\tilde{J}$ be the ideal in $\CC[x_0,\ldots,x_{l_{\PP}}, w_{s_X+1},\ldots,w_k]$
generated by
\begin{equation*}
(x_0,\ldots,x_{l_{\PP}})^{d_X},
\frac{\partial \tilde{F}}{\partial x_0},\ldots, \frac{\partial \tilde{F}}{\partial x_{l_{\PP}}},
\tilde{f}_1,\ldots,\tilde{f}_{s_X},
\end{equation*}
so that
$$
S/J'\cong \CC[x_0,\ldots,x_{l_{\PP}}, w_{s_X+1},\ldots,w_k]/\tilde{J}.
$$
Also, define $\bar{R}^0$ as the quotient of $\CC[x_0,\ldots,x_{l_{\PP}}]$ by the ideal generated by
$$
(x_0,\ldots,x_{l_{\PP}})^{d_X}, \tilde{f}_1,\ldots,\tilde{f}_{s_X},
$$
and let $\bar{J}^0$ be the ideal in $\bar{R}^0[w_{s_X+1},\ldots,w_k]$ generated by
$$
\frac{\partial \tilde{F}}{\partial x_0},\ldots, \frac{\partial \tilde{F}}{\partial x_{l_{\PP}}}.
$$
We have isomorphisms of bigraded algebras
\begin{equation*}%\label{eq:algebras}
\bar{R}\cong\CC[x_0,\ldots,x_{l_{\PP}}, w_{s_X+1},\ldots,w_k]/\tilde{J}\cong\bar{R}^0[w_{s_X+1},\ldots,w_k]/\bar{J}^0.
\end{equation*}

Note that the ideal $\bar{J}^0$ is generated by $l_{\PP}+1$ polynomials, and each of them
is linear in the variables $w_i$ and has degree $d_X-1$ in variables $x_0,\ldots,x_{l_{\PP}}$.
Consider the grading on the algebra $\CC[w_{s_X+1},\ldots,w_k]$ such that the variables $w_j$ have degree $1$.
For all non-negative integers $p$ and $r$ we have
$$
\dim\bar{R}_{p,r-pd_X}=\dim \bar{R}^0_{r}\cdot\dim \CC[w_{s_X+1},\ldots,w_k]_{p}= \dim \bar{R}^0_{r}\cdot\binom{p+k-s_X-1}{p}
$$
if $r\le d_X-2$, and
\begin{multline*}
\dim\bar{R}_{p,d_X-1-pd_X}=\\=\dim \bar{R}^0_{d_X-1}\cdot\dim \CC[w_{s_X+1},\ldots,w_k]_{p}-(l_{\PP}+1)\cdot \dim \CC[w_{s_X+1},\ldots,w_k]_{p-1}=\\=
\dim \bar{R}^0_{d_X-1}\cdot\binom{p+k-s_X-1}{p}- (l_{\PP}+1)\cdot\binom{p+k-s_X-2}{p-1}.
\end{multline*}
On the other hand, by Lemma~\ref{lemma:dim-polynomials} one has
$$
\dim \bar{R}^0_{r}\ge \binom{r+l_{\PP}-s_X}{r},
$$
and moreover
$$
\dim \bar{R}^0_{r}\ge \binom{r+l_{\PP}-s_X}{r}+s_X
$$
if $r>0$.
Since $d_X\ge 2$, we have in particular
$$
\dim \bar{R}^0_{d_X-1}\ge \binom{d_X-1+l_{\PP}-s_X}{d_X-1}+s_X.
$$
Note that
$$
\dim R_{p_X,-i_X}=\dim R_{p_X,r_X-p_Xd_X}\ge \dim\bar{R}_{p_X,r_X-p_Xd_X}.
$$
Therefore, we have
$$
\dim R_{p_X,-i_X}\ge \binom{r_X+l_{\PP}-s_X}{r_X}\cdot\binom{p_X+k-s_X-1}{p_X}
$$
if $r_X<d_X-1$, so that the first assertion of the lemma is implied by Theorem~\ref{theorem:middle-Hodge-numbers}.
In the case if $r_X=d_X-1$, we get
\begin{multline*}
\dim R_{p_X,-i_X}\ge\\ \ge \binom{p_X+k-{s_X}-1}{p_X}\cdot\left(\binom{{d_X}-1+{l_{\PP}}-{s_X}}{{d_X}-1}+{s_X}\right)
-({l_{\PP}}+1)\cdot \binom{p_X+k-{s_X}-2}{p_X-1}.
\end{multline*}
Since
$$
\binom{p_X+k-{s_X}-1}{p_X}>\binom{p_X+k-{s_X}-2}{p_X-1},
$$
we have
$$
\dim R_{p_X,-i_X}>\binom{p_X+k-{s_X}-1}{p_X}\cdot
\left(\binom{{d_X}-1+{l_{\PP}}-{s_X}}{{d_X}-1}+{s_X}-l_{\PP}-1\right)
$$
when ${r_X}={d_X}-1$.
Now Theorem~\ref{theorem:middle-Hodge-numbers} implies the second assertion of the lemma as well.
\end{proof}

\begin{lemma}
\label{lemma:equal 1}
Let $X$ be a smooth well formed Fano weighted
complete intersection in $\P$.
Suppose that $k={s_X}+1$ and ${r_X}=0$. Then $h^{p_X,n-p_X}_{pr}(X)=1$.
\end{lemma}

\begin{proof}
Bidegree estimates show that
the bigraded component of bidegree $(p_X,-i_X)$ in
the algebra $S=\CC[x_0,\ldots,x_N,w_1,\ldots,w_k]$
is generated by
the monomial~$w_k^{p_X}$. On the other hand, one obviously has $w_k^{p_X}\notin J$,
{and hence the bigraded component $R_{p_X,-i_X}$ is one-dimensional. Thus the assertion follows from Theorem~\ref{theorem:middle-Hodge-numbers}.}
\end{proof}

Now we are able to deduce the following positivity result.

\begin{corollary}\label{corollary:greater-than-1}
Let $X$ be a smooth well formed Fano weighted
complete intersection
in~$\P$ which is not an intersection with a linear cone.
Suppose that $X$ is not a complete intersection of quadrics
in $\P=\P^N$. Then~\mbox{$h_{pr}^{p_X,n-p_X}(X)$} is positive.
Moreover, one has~\mbox{$h_{pr}^{p_X,n-p_X}(X)=1$} if and only if $k={s_X}+1$ and~\mbox{${r_X}=0$}.
\end{corollary}
\begin{proof}
Note that ${d_X}>2$ by Lemma~\ref{lemma:only-quadrics}(i), and ${s_X}<k$ by definition.
Also, one has $l_{\PP}\ge k$ by Theorem~\ref{theorem:Fano-invariants}(iii).
In particular, we have
\begin{equation}\label{eq:second-factor}
\binom{{r_X}+l_{\PP}-{s_X}}{{r_X}}\ge 1,
\end{equation}
and the equality holds if and only if ${r_X}=0$. Since $p_X>0$, we also have
\begin{equation}\label{eq:first-factor}
\binom{p_X+k-{s_X}-1}{p_X}\ge 1,
\end{equation}
and the equality holds if and only if $k={s_X}+1$.

Suppose that ${r_X}<{d_X}-1$.
Then
$$
h_{pr}^{p_X,n-p_X}(X)\ge \binom{p_X+k-{s_X}-1}{p_X}\cdot\binom{{r_X}+{l_{\PP}}-{s_X}}{{r_X}}
$$
by Lemma~\ref{lemma:s1<k}.
By~\eqref{eq:second-factor} and~\eqref{eq:first-factor}
this gives $h_{pr}^{p_X,n-p_X}(X)\ge 1$, and the equality can hold only
if $k={s_X}+1$ and ${r_X}=0$. On the other hand, if $k={s_X}+1$ and ${r_X}=0$, then~\mbox{$h_{pr}^{p_X,n-p_X}(X)=1$} by Lemma~\ref{lemma:equal 1}.

Now suppose that ${r_X}={d_X}-1$. Then
$$
h_{pr}^{p_X,n-p_X}(X)>\binom{p_X+k-{s_X}-1}{p_X}\cdot
\left(\binom{{d_X}-1+{l_{\PP}}-{s_X}}{{d_X}-1}+{s_X}-l_{\PP}-1\right)
$$
Furthermore,
since ${d_X}\ge 3$, we have
$$
\binom{{d_X}-1+{l_{\PP}}-{s_X}}{{d_X}-1}+{s_X}-{l_{\PP}}-1\ge ({d_X}-1+{l_{\PP}}-{s_X})+{s_X}-{l_{\PP}}-1={d_X}-2\ge 1.
$$
Therefore, keeping in mind inequality~\eqref{eq:first-factor}, we obtain
$h_{pr}^{p_X,n-p_X}(X)>1$ in this case.
\end{proof}

\begin{lemma}\label{lemma:vanishing}
Let $X$ be a smooth well formed Fano weighted
complete intersection
in $\P$.
Then $h_{pr}^{q,n-q}(X)=0$ for every $q<p_X$.
\end{lemma}

\begin{proof}
Suppose that $h_{pr}^{q,n-q}(X)>0$ for some $q<p_X$.
Then Theorem~\ref{theorem:middle-Hodge-numbers} implies that there is a monomial
$$
m_w\cdot m_x\in S=\CC[x_0,\ldots,x_N,w_1,\ldots,w_k]
$$
of bidegree~\mbox{$(q,-i_X)$}, where
$m_w$ is a monomial in~$w_i$
of degree~$q$, and~$m_x$ is a monomial in~$x_j$.
Let~\mbox{$(q,-t)$} be the bidegree of~$m_w$. Then $t\le q{d_X}$, and the bidegree of~$m_x$ is~\mbox{$(0,t-i_X)$}. But since~\mbox{$q\le p_X-1$}, one has
$$
t-i_X\le q{d_X}-i_X\le  (p_X-1){d_X}-i_X={r_X}-{d_X}<0,
$$
which gives a contradiction.
\end{proof}

\begin{lemma}
\label{lemma:int-of-quadrics}
Let $X\subset\P^N$ be a smooth Fano complete intersection of $k$ quadrics
of dimension
$n=N-k$ in $\P^N$.
The following assertions hold.
\begin{itemize}
\item[(i)] Suppose that
$k=1$, so that $X$ is a quadric hypersurface.
Then~\mbox{$h^{q,n-q}_{pr}(X)=0$} for all $q$,
with the only exception $h^{\frac{n}{2},\frac{n}{2}}_{pr}(X)=1$ for even dimension~$n$.

\item[(ii)]
Suppose that $k>1$. Then
$h^{q,n-q}_{pr}(X)=0$ for $q<p_X$. Moreover, one has
$$
h^{p_X,n-p_X}_{pr}(X)= \binom{\frac{N-1}{2}}{k-1}>1
$$
for even $i_X=N-2k+1$, and
$$
h^{p_X,n-p_X}_{pr}(X)\ge k\cdot \left(\binom{\frac{N}{2}}{\frac{N}{2}-k+1}-\binom{\frac{N}{2}-1}{\frac{N}{2}-k}\right)>1.
$$
for odd $i_X$.
\end{itemize}
\end{lemma}

\begin{proof}
Assertion~(i) is obvious and well known.

Let us prove assertion~(ii).
The vanishing of $h_{pr}^{q,n-q}(X)$ for all $q<p_X$ follows from Lemma~\ref{lemma:vanishing}.
For the estimates from below we will
use Theorem~\ref{theorem:middle-Hodge-numbers}.
Similarly to the proof of Lemma~\ref{lemma:s1<k}, one can assume that $X$ is of Fermat type.
We have $i_X=n-k+1$ and~\mbox{$p_X=\lceil\frac{n-k+1}{2}\rceil$}.

Suppose that $i_X$ is even, so that $p_X=\frac{n-k+1}{2}$.
Let $\bar{J}$ be the ideal in $R$ generated by the variables $x_0,\ldots,x_N$.
Then there is a natural surjective map of bigraded algebras
$$
R\to R/\bar{J}\cong\CC[w_1,\ldots,w_k].
$$
Note that the kernel of this map has zero intersection with $R_{p_X,-i_X}$, and consider the
new grading on $\CC[w_1,\ldots,w_k]$ such that the variables $w_1,\ldots,w_k$ have degree $1$. One has
$$
h^{p_X,n-p_X}_{pr}(X)=\dim R_{p_X,-i_X}=\dim \CC[w_1,\ldots,w_k]_{p_X}=\binom{\frac{n+k-1}{2}}{k-1}=\binom{\frac{N-1}{2}}{k-1}>1.
$$

Now suppose that $i_X$ is odd, so that $p_X=\frac{n-k}{2}+1$.
Starting from the complete intersection $X$ written in the Fermat form~\eqref{eq:Fermat}
and making linear changes of coordinates if necessary, we can assume that the quadric polynomials generating the homogeneous ideal of $X$
have the form
$$
f_j=x_{j-1}^2+\sum\limits_{i\ge k} \alpha_{j,i}x_i^2, \quad 1\le j\le k.
$$
Let $\bar{J}$ be the ideal in $R$ generated by
$$
(x_0,\ldots,x_{k-1})^2, x_k,\ldots,x_N,
$$
where $(x_0,\ldots,x_{k-1})^2$ is the ideal generated by all monomials of degree $2$ in $x_0,\ldots,x_{k-1}$.
Let $\tilde{J}$ be the ideal in $\CC[x_0,\ldots,x_{k-1},w_1,\ldots,w_k]$ generated by $(x_0,\ldots,x_{k-1})^2$ and $k$ monomials
$$
w_1x_0,\ldots,w_kx_{k-1}.
$$
Then there is a natural surjective map of bigraded algebras
$$
R\to R/\bar{J}\cong\CC[x_0,\ldots,x_{k-1},w_1,\ldots,w_k]/\tilde{J}.
$$
Thus the dimension
$\dim R_{p_X,-i_X}=\dim R_{p_X,1-2p_X}$
is bounded from below by the dimension
of the bigraded component of bidegree $(p_X,1-2p_X)$ of the algebra~\mbox{$\CC[x_0,\ldots,x_{k-1},w_1,\ldots,w_k]/\tilde{J}$}.
Define the grading on $\CC[w_1,\ldots,w_k]$ so that the degree of the variables
$w_1,\ldots,w_k$ equals $1$.
One has
\begin{multline*}
\dim R_{p_X,-i_X}\ge\\ \ge
\dim\CC[x_0,\ldots,x_{k-1},w_1,\ldots,w_k]_{p_X,1-2p_X}-k\cdot\dim\CC[w_1,\ldots,w_k]_{p_X-1}=
\\=\dim\CC[x_0,\ldots,x_{k-1}]_{1}\cdot\dim\CC[w_1,\ldots,w_k]_{p_X}-k\cdot\dim\CC[w_1,\ldots,w_k]_{p_X-1}=\\
=k\big(\dim\CC[w_1,\ldots,w_k]_{p_X}-\dim\CC[w_1,\ldots,w_k]_{p_X-1}\big).
\end{multline*}
Therefore, we have
\begin{multline*}
h^{p_X,n-p_X}_{pr}(X)=\dim R_{p_X,-i_X}\ge k\cdot \left(\binom{k+p_X-1}{p_X}-\binom{k+p_X-2}{p_X-1}\right)=\\=
k\cdot\left(\binom{\frac{N}{2}}{\frac{N}{2}-k+1}-\binom{\frac{N}{2}-1}{\frac{N}{2}-k}\right)\ge k>1.  \qedhere
\end{multline*}
\end{proof}

Recall that if $X$ is a smooth Fano complete intersection
of quadrics and at least one cubic in~$\P^N$,
then
$$
r_X=3\left\lceil\frac{i_X}{3}\right\rceil-i_X\in\{0,1,2\}.
$$

We summarize the results of this section
in terms of the Hodge level~$\hh(X)$, see
Definition~\ref{definition: Hodge level}.

\begin{corollary}[Proposition \ref{proposition:Hodge-level}]
\label{corollary:h-n-2p}
Let $X$ be a smooth well formed Fano weighted
complete intersection which is not an intersection with a linear cone.
If $X$ is an odd-dimensional quadric, then $\hh(X)=0$.
Otherwise~\mbox{$\hh(X)=n-2p_X$}.
\end{corollary}

\begin{proof}
If $X$ is a complete intersection of quadrics in $\P^N$, then
the assertion follows from Lemma~\ref{lemma:int-of-quadrics}.
Otherwise it is given by
Corollary~\ref{corollary:greater-than-1} and Lemma~\ref{lemma:vanishing}.
\end{proof}

\section{Proofs of the main results}
\label{section:proofs}

In this section we prove Theorems~\ref{theorem:main}
and~\ref{theorem:main-CY}, Propositions~\ref{proposition:main-CY} and~\ref{proposition:Hodge-level},
and Corollary~\ref{corollary:Hodge-level}.
We use the notation introduced in the beginning
of~\S\ref{section:bounds}.

\begin{remark}\label{remark:Kodaira}
Let $X$ be a smooth $n$-dimensional Fano variety.
We always have $h^{0,n}(X)=0$ by Kodaira vanishing.
This implies that $X$ is always diagonal provided that $n\le 2$,
and it is always of curve type provided that $n=3$.
\end{remark}

We start with the case of complete intersections of quadrics and cubics in the usual projective space.

\begin{lemma}\label{corollary:int-of-quadrics}
Let $X\subset\P^N$ be a smooth $n$-dimensional Fano complete intersection of quadrics and
cubics (including the case when there are either no quadrics or no cubics).
Then
\begin{itemize}
\item[(i)] $X$ is $\QQ$-homologically minimal if and only if $X$ is an odd-dimensional quadric;

\item[(ii)] $X$ is diagonal if and only if $X$ is a quadric, or an even-dimensional intersection of two quadrics,
or a cubic surface;

\item[(iii)] $X$ is of curve type if and only if $X$ is either an odd-dimensional intersection of at most three quadrics, or a threefold, or a cubic fivefold;

\item[(iv)] $X$ is of $2$-Calabi--Yau type if and only if $X$ is a cubic fourfold;

\item[(v)] $X$ is of $3$-Calabi--Yau type if and only if $X$ is either a seven-dimensional cubic or a five-dimensional intersection of a quadric and a cubic.
\end{itemize}
\end{lemma}
\begin{proof}
Assertion (i) immediately follows from Corollary~\ref{corollary:h-n-2p}.
In what follows we will assume that $X$ is not an odd-dimensional quadric.
Let $X$ be a complete intersection of $k_1$ quadrics and $k_2$ cubics.
Note that in every case the sufficiency of the provided conditions is straightforward to check.

One has
$$
p_X=\left\lceil \frac{n-k_1+1}{2}\right\rceil
$$
if $k_2=0$ and
$$
p_X=\left\lceil\frac{n-k_1-2k_2+1}{3}\right\rceil
$$
otherwise.
This is less or equal to $\frac{n}{2}$, and the equality holds if and only if $X$ is a quadric, an even-dimensional intersection of two quadrics,
or a cubic surface. Thus, using Corollary~\ref{corollary:h-n-2p}, we obtain assertion~(ii).

Suppose that $X$ is of curve type.
Then $n$ is odd, and
it follows from Corollary~\ref{corollary:h-n-2p} that
$p_X=\frac{n-1}{2}$. If $k_2=0$, we conclude that $X$ is a complete intersection of at most three quadrics.
If $k_2>0$, then $X$ is either a threefold, or a cubic fivefold. This gives assertion~(iii).

Now suppose that $X$ is of $2$-Calabi--Yau type or of $3$-Calabi--Yau type.
We know from Lemma~\ref{lemma:int-of-quadrics} that $k_2>0$.
Hence it follows from Corollary~\ref{corollary:greater-than-1} that

${r_X}=0$ and $k_2=1$.
If $n$ is even (so that $X$ is of $2$-Calabi--Yau type), then
$$
p_X=\frac{n}{2}-1
$$
by Corollary~\ref{corollary:h-n-2p}.
This implies
\[
\frac{n-k_1-2k_2+1}{3}=\frac{n}{2}-1,
\]
which means that $X$ is a cubic fourfold.
If $n$ is odd (so that $X$ is of $3$-Calabi--Yau type), then
$$
p_X=\frac{n-3}{2}
$$
by Corollary~\ref{corollary:h-n-2p}.
This implies
\[
\frac{n-k_1-2k_2+1}{3}=\frac{n-3}{2},
\]
which means that $X$ is either a seven-dimensional cubic or a five-dimensional complete intersection of a quadric and a cubic.
Therefore, we obtain assertions~(iv) and~(v).
\end{proof}

\begin{lemma}\label{lemma:d-4}
Let $X$ be a smooth well formed Fano weighted complete intersection of dimension $n\ge 3$
in $\P$ which is not an intersection with a linear cone.
Suppose that~\mbox{${d_X}\ge 4$}. Then
\begin{itemize}
\item[(i)] $X$ is not diagonal (and in particular not $\QQ$-homologically minimal);

\item[(ii)] $X$ is of curve type if and only if $n=3$;

\item[(iii)] $X$ is not of $2$-Calabi--Yau type;

\item[(iv)] $X$ is of $3$-Calabi--Yau type if and only if $X$ is a five-dimensional quartic hypersurface in $\P(1^6,2)$.
\end{itemize}
\end{lemma}
\begin{proof}
Recall that $i_X\le n$ by Theorem~\ref{theorem:low-coindex}(ii). Therefore, by Corollary~\ref{corollary:h-n-2p} one has
\begin{equation}\label{eq:h}
\hh(X)=n-2\left\lceil\frac{i_X}{{d_X}}\right\rceil\ge n-2\left\lceil\frac{n}{4}\right\rceil.
\end{equation}
Since $n\ge 3$, we see that $\hh(X)>0$, which implies assertion~(i).

To prove the remaining assertions, we
can restrict to the case $n>5$.
Indeed, for~\mbox{$n=4$} and $n=5$ it follows from Remark~\ref{remark:Kodaira} and the classification of such varieties
(see~\mbox{\cite[\S5]{PrzyalkowskiShramov-Weighted}}, cf.~\mbox{\cite[Proposition~2.2.1]{Kuchle-Geography}})
that $X$ is not of curve type and not of $2$-Calabi--Yau type;
moreover,~$X$ is of $3$-Calabi--Yau type if and only if~$X$ is a five-dimensional quartic hypersurface in~\mbox{$\P(1^6,2)$}.
We deduce from~\eqref{eq:h} that
$$
\hh(X)\ge \frac{n}{2}-2,
$$
which means that for $n>5$ the variety $X$ is not of curve type.
This proves assertion~(ii).

Now suppose that $X$ is of $m$-Calabi--Yau type for some $m$.
Then we know from Corollary~\ref{corollary:greater-than-1} that $i_X$ is divisible by ${d_X}$,
so that~\eqref{eq:h} reads
$$
m=\hh(X)= n-\frac{2i_X}{{d_X}}\ge n-\frac{2n}{4}= \frac{n}{2}.
$$
This implies assertion~(iii), and shows that one can have $\hh(X)=3$ only for
$n=6$. However, if $X$ is of $3$-Calabi--Yau type, its dimension must be odd by definition,
and thus we obtain assertion~(iv) as well.
\end{proof}

Now we are ready to prove our main results.

\begin{proof}[Proof of Theorem~\ref{theorem:main} and Proposition~\ref{proposition:main-CY}]
By Lemma~\ref{lemma:small-dimension}
we can assume that~\mbox{$\dim X\ge 3$}. Therefore, by Lemma~\ref{lemma:d-4}
we can assume that $d_X\le 3$. According to Lemma~\ref{lemma:only-quadrics}, this means
that $X$ is a complete intersection of quadrics and cubics in~\mbox{$\P=\P^N$}.
Now the required assertions follow from Lemma~\ref{corollary:int-of-quadrics}.
\end{proof}

\begin{proof}[Proof of Theorem~\ref{theorem:main-CY}]
The Hodge level $\hh(X)$ is computed by Corollary~\ref{corollary:h-n-2p}. Thus,
the assertion follows from Corollary~\ref{corollary:greater-than-1}
and Lemma~\ref{lemma:int-of-quadrics}.
\end{proof}

Note also that one can obtain an alternative proof of Proposition~\ref{proposition:main-CY}
using Theorem~\ref{theorem:main-CY}.

The assertion of Proposition~\ref{proposition:Hodge-level} is given by Corollary~\ref{corollary:h-n-2p}.
Finally, we prove Corollary~\ref{corollary:Hodge-level}.

\begin{proof}[Proof of Corollary~\ref{corollary:Hodge-level}]
Suppose that the assumptions of assertion~(i) hold. Then $X$ is not an
odd-dimensional quadric.
Hence
$$
\hh(X)= n-2\left\lceil\frac{i_X}{{d_X}}\right\rceil
$$
by Corollary~\ref{corollary:h-n-2p}.
In particular, one has $\hh(X)=n-2$ if $i_X\le 2$.
If $i_X\ge 3$ and $X$ is not a complete intersection of quadrics, then
$d_X\ge 3$ by Lemma~\ref{lemma:only-quadrics}(i), and one also has~\mbox{$\hh(X)=n-2$}. If $i_X\ge 4$ and $X$ is not a complete intersection of quadrics and cubics, then
$d_X\ge 4$ by Lemma~\ref{lemma:only-quadrics}, and again
$\hh(X)=n-2$.

Now suppose that $X$ is not a complete intersection of quadrics in a projective space.
Then~\mbox{${d_X}\ge 3$} by Lemma~\ref{lemma:only-quadrics}(i).
Therefore, Corollary~\ref{corollary:h-n-2p}  and Theorem~\ref{theorem:low-coindex}(ii)
give
$$
\hh(X)= n-2\left\lceil\frac{i_X}{{d_X}}\right\rceil\ge n-2\left\lceil\frac{n}{3}\right\rceil.
$$
Since $\left\lceil\frac{n}{3}\right\rceil\le\frac{n+2}{3}$, we obtain
$$
\hh(X)\ge n-2\left(\frac{n+2}{3}\right)=\frac{n-4}{3},
$$
which gives assertion~(ii).
\end{proof}

\section{Quasi-smooth complete intersections}
\label{section:quasi-smooth}

In this section we briefly discuss quasi-smooth weighted complete intersections.

\begin{definition}[{see~\cite[Definition 6.3]{IF00}}]
\label{definition: quasi-smoothness}
Let $\pi\colon \mathbb A^{N+1}\setminus \{0\}\to \P$ be the natural projection to a weighted projective space. A subvariety $X\subset \P$
is called \emph{quasi-smooth} if $\pi^{-1}(X)$ is smooth.
\end{definition}

The following result allows one to check quasi-smoothness of weighted hypersurfaces in terms of weights and degrees.

\begin{theorem}[{\cite[Theorem 8.1]{IF00}}]
\label{theorem:quasi-smooth hypersurfaces}
Let $X$ be a general hypersurface in $\P(a_0,\ldots,a_N)$ of degree $d$, which is not an intersection with a linear cone.
Then $X$ is quasi-smooth if and only if for any $0\le s\le N$ and for any non-empty subset
$$
I=\{i_0,\ldots,i_s\}\subset \{0,\ldots,N\}
$$
there either exists a monomial $x_{i_0}^{m_0}\cdot\ldots\cdot x_{i_s}^{m_s}$ of degree $d$,
or for any $\mu=1,\ldots,k+1$ there exist monomials $x_{i_0}^{m_{0,e_\mu}}\cdot\ldots\cdot x_{i_{s}}^{m_s,e_\mu} x_{e_\mu}$
of degree $d$ for $s+1$ distinct variables $e_\mu$.
\end{theorem}

Note that there is a pure Hodge structure on the cohomology of quasi-smooth weighted complete
intersections.
Theorem~\ref{theorem:middle-Hodge-numbers} actually holds in the quasi-smooth case, see~\mbox{\cite[Theorem~3.6]{Ma99}}.

Let us consider some examples from the Alexeev's list (see Remark~\ref{remark:quasismooth}).

\begin{example}
\label{example:quasismooth-4-1}
Let $X$ be a general hypersurface of degree $4$ in $\P(1,1,1,1,2,2)$. Then $X$ is a well formed Fano hypersurface, and $X$ is quasi-smooth by Theorem~\ref{theorem:quasi-smooth hypersurfaces}.
One has
$$
h^{1,3}(X)=\dim\big(R_{1,-4}\big)=1
$$
by Theorem~\ref{theorem:middle-Hodge-numbers}.
Thus, $X$ is of K3 type. Moreover,
by~\mbox{\cite[Corollary~4.2]{Kuz15b}}, the derived category of coherent sheaves on $X$ has a semiorthogonal decomposition on
a category of~K3 type and three exceptional objects.
\end{example}

\begin{example}
\label{example:quasismooth-6}
Let $X$ be a general hypersurface of degree $16$ in $\P(1, 4, 5, 6, 8, 8, 8, 8)$.  Then~$X$ is a well formed Fano hypersurface, and $X$ is quasi-smooth by Theorem~\ref{theorem:quasi-smooth hypersurfaces}.
One has
$$
h^{2,4}(X)=\dim\big(R_{2,-32}\big)=1
$$
by Theorem~\ref{theorem:middle-Hodge-numbers}.
Thus, $X$ is of K3 type. Moreover,
in the same way as in Example~\ref{example:quasismooth-4-1}, the derived category of coherent sheaves on $X$
contains a category of K3 type.
\end{example}

\begin{example}
\label{example:quasismooth-4-2}
Let $X$ be a hypersurface of degree $6$ in $\P(1,1,1,3,3,3)$. Then~$X$ is a well formed Fano hypersurface, and $X$ is quasi-smooth by Theorem~\ref{theorem:quasi-smooth hypersurfaces}.
One has
$$
h^{1,3}(X)=\dim\big(R_{1,-6}\big)=1
$$
by Theorem~\ref{theorem:middle-Hodge-numbers}.
Thus, $X$ is of K3 type. Moreover,
in the same way as in Example~\ref{example:quasismooth-4-1}, the derived category of coherent sheaves on $X$
contains a category of K3 type.
\end{example}

\section{Discussion}
\label{section:discussion}

In this section we discuss some remaining questions and problems.

Theorem~\ref{theorem:main} gives non-vanishing of some middle Hodge numbers for Fano weighted complete intersections.
This non-vanishing holds for trivial reasons for Calabi--Yau varieties, since for $n$-dimensional Calabi--Yau
variety $X$ one has $h^{0,n}(X)=1$. Calculations in some examples shows that
$h^{0,n}(X)>0$ for many smooth $n$-dimensional weighted complete intersections $X$ of general type.
On the other hand, if one drops the assumption that $X$ is a weighted complete intersection, there are plenty of examples
of $\QQ$-homologically minimal surfaces of general type known as \emph{fake projective planes}, see, for example,~\cite{Mu79},~\cite{CS10}.
Also, there are examples of $\QQ$-homologically minimal fourfolds of general type, see~\cite{PG06}.

\begin{question}
Do there exist smooth weighted complete intersections of general type that are $\QQ$-homologically minimal, diagonal, of curve type, or of $m$-Calabi--Yau type for small~$m$? Is it true that $h^{0,n}(X)>0$ for all smooth $n$-dimensional weighted complete intersections~$X$ of general type?
\end{question}

In all examples of smooth Fano weighted complete intersections we have considered middle Hodge numbers
that are closer to the center of the Hodge diamond are larger than those that are further from the center: for
an $n$-dimensional variety $X$ one has~\mbox{$h^{p,n-p}(X)\leqslant h^{q,n-q}(X)$} if~\mbox{$p<q\le \frac{n}{2}$}.
This is not true for some Calabi--Yau varieties: there are \emph{rigid Calabi--Yau threefolds}, that are ones with $h^{1,2}(X)=0$, while
$h^{3,0}(X)=1$, see, for example,~\cite{Og96}. However we do not know rigid Calabi--Yau weighted complete intersections.

\begin{question}
Let $X$ be a smooth $n$-dimensional weighted complete intersection. Is it true that $h^{p,n-p}(X)\leqslant h^{q,n-q}(X)$ if $p<q\le \frac{n}{2}$?
If not, is this true if we also assume that~$X$ is a Fano variety, or a Calabi--Yau variety?
\end{question}

There is a generalization of the construction of Fano varieties as complete intersections.
Namely, one can consider zeros of
sections of equivariant vector bundles on Grassmannians or partial flag manifolds,
 see, for instance,~\cite{Kuchle-Grass},~\cite{Kuz15},~\cite{Kuz16}.

\begin{problem}
Bound Hodge numbers and compute Hodge level for such varieties.
\end{problem}

From the point of view of classification of Fano varieties
the most complicated and interesting case is the case of Fano index $1$ and Picard rank $1$.
By Corollary~\ref{corollary:Hodge-level}(i), for an $n$-dimensional weighted complete intersection $X$
of this type,
one has
$\mbox{$h^{1,n-1}(X)>0$}$. There is an approach to computing this number via Mirror Symmetry.
That is, one can construct Calabi--Yau compactified (toric) Landau--Ginzburg models
for a big class of Fano varieties (see~\cite{Prz13},~\cite{PSh14},~\cite{Prz16},~\cite{Prz17}),
cf. also the notion of a tame compactified Landau--Ginzburg model in~\cite{KKP17}.
Given such a Landau--Ginzburg model~$LG_X$, one can define the number $k_{LG_X}$ as the number
of irreducible components of reducible fibers of $LG_X$ minus the number of the reducible fibers.
It is expected that
one has~\mbox{$k_{LG_X}=h^{1,n-1}(X)$}. This
is proved for Picard rank $1$ Fano threefolds in~\cite{Prz13} and for Fano complete intersections
in usual projective spaces in~\cite{PSh15a}. 

This approach can be applied to any smooth Fano variety, not necessary to a weighted complete intersection,
provided its weak Landau--Ginzburg model is known. Most of the known constructions of Picard rank $1$ Fano varieties
are via complete intersections in weighted projective spaces or Grassmannians.
Fortunately, there are constructions of weak Landau--Ginzburg models for some of them, for instance, for weighted
complete intersections of Cartier divisors (see~\cite{Prz11}) or of two hypersurfaces (see~\cite{PSh17b}),
and for complete intersections in Grassmannians (see~\cite{BCFKS98},~\cite{PSh14},~\cite{PSh15b},~\cite{PSh17a});
these constructions are based on Givental's approach~\cite{Gi97}.

\begin{problem}
Construct a Calabi--Yau compactification $LG_X$ of a weak Landau--Ginzburg model for a smooth Fano weighted complete intersection $X$
of dimension $n$ and compute $k_{LG_X}$.
The fact that $k_{LG_X}>0$ (that is, that the Landau--Ginzburg model has a reducible fiber)
is conjecturally equivalent to the fact that~\mbox{$\hh(X)=n-2$}.
\end{problem}

\begin{remark}
\label{remark:minimality}
Except for $\QQ$-homological minimality there are other notions of minimality for Fano varieties. For instance, in~\cite{Prz07b} (see also~\cite{Prz07a}) the notion of \emph{quantum minimality} was introduced:
a Fano variety is quantum minimal if the subring in cohomology generated by the anticanonical class is
closed with respect to quantum multiplication; in other words, it contains the minimal possible subring of small quantum cohomology
ring containing the anticanonical class. Note that by~\mbox{\cite[Lemma~5.5]{Ga02}}
(see also~\cite{Prz07a}) all smooth Fano weighted complete intersections
are quantum minimal.
\end{remark}

\appendix
\section{Smooth Fano weighted complete intersections of small dimension}
\label{section:tables}

In Tables~\ref{table:dim 2} and~\ref{table:dim 3}
we list smooth well formed Fano weighted complete intersections that are not
intersections with linear cones and have dimensions $2$ and $3$, respectively; according to the convention of~\S\ref{section:preliminaries},
we exclude projective spaces from these lists. The
classification of del Pezzo surfaces is well known, as well as the classification
of Picard rank 1 Fano threefolds (see~\cite{Is77}). On the other hand,
by the Lefschetz-type theorem smooth well formed weighted complete intersections of dimension at least $3$ in weighted
projective spaces have Picard rank~$1$, see~\cite[Proposition~1.4]{Ma99},
so the three-dimensional case indeed can be extracted from~\cite{Is77}.
An alternative way to compile our lists is to use the bounds for discrete invariants of
weighted projective spaces and Fano weighted complete intersections therein provided
by Theorem~\ref{theorem:Fano-invariants}.
An advantage of this approach is that one does not need to check that certain Fano varieties cannot be re-embedded
into weighted projective spaces as weighted complete intersections.

For every line in each of the two tables, in the first column we put the ambient weighted projective space,
in the second one we list degrees
of hypersurfaces that determine our weighted complete intersection, and in the third one we put the
unique non-trivial Hodge number, that is, $h^{1,1}$ in the two-dimensional case and $h^{1,2}$ in the three-dimensional
case.

\begin{table}[h]
\centering
\begin{tabular}{||c|c|c||}
\hline
\hline
$\P$ & Degrees & $h^{1,1}$ \\
\hline
\hline
$\P(1^2,2,3)$ & 6 & $9$ \\
\hline
$\P(1^3,2)$ & 4 & $8$ \\
\hline
$\P^3$ & 3 & $7$ \\
\hline
$\P^4$ & 2,2 & $6$ \\
\hline
$\P^3$ & 2 & $2$ \\
\hline
\hline
\end{tabular}
\caption{Del Pezzo surfaces}\label{table:dim 2}
\end{table}

\begin{table}[h]
\centering
\begin{tabular}{||c|c|c||}
\hline
\hline
$\P$ & Degrees & $h^{1,2}$ \\
\hline
\hline
$\P(1^4,3)$ & 6 & $52$ \\
\hline
$\P^4$ & 4 & $30$ \\
\hline
$\P^5$ & 2,3 & $20$ \\
\hline
$\P^6$ & 2,2,2 & $14$ \\
\hline
$\P(1^3,2,3)$ & 6 & $21$ \\
\hline
$\P(1^4,2)$ & 4 & $10$ \\
\hline
$\P^4$ & 3 & $5$ \\
\hline
$\P^5$ & 2,2 & $2$ \\
\hline
$\P^4$ & 2 & $0$ \\
\hline
\hline
\end{tabular}
\caption{Fano threefolds}\label{table:dim 3}
\end{table}

\end{document}